\documentclass[12pt]{article}
\usepackage[intlimits]{amsmath}
\usepackage{amsfonts,amssymb,amscd,amsthm,cite}
\input xy
\xyoption{all}

\def\supp{\operatorname{supp}}
\def\vol{\operatorname{vol}}
\def\re{\operatorname{Re}}
 
\newcommand{\im}{\mathop{\rm Im}}

\def\tg{\operatorname{tan}}
\def\ctg{\operatorname{cot}}

\newtheorem{theorem}{Theorem}[section]
\newtheorem{proposition}{Proposition}[section]
\newtheorem{lemma}{Lemma}[section]

\theoremstyle{definition}

\newtheorem{definition}{Definition}[section]

\newtheorem{remark}{Remark}[section]
\newtheorem{condition}{Condition}[section]

\def\im{\operatorname{Im}}

 \usepackage{graphicx}
\usepackage{epstopdf} 
 
\begin{document}

\title{On traces of operators, associated with actions of compact Lie groups}

\author{Savin A.Yu., Sternin B.Yu.}
\date{}

\maketitle

\begin{abstract}
 Given a pair  $(M,X)$, where $X$ is a smooth submanifold in a closed smooth manifold $M$,
 we study the operation, which takes each operator $D$ on the ambient manifold to  a certain
 operator on the submanifold. The latter operator is called the trace of $D$. More precisely, we study traces of operators, associated with actions of compact Lie groups  on $M$. We show that traces of such operators are localized at special submanifolds in $X$ and study the structure of the  traces on these submanifolds.  
\end{abstract}

\section{Statement of the problem }

Let $(M,X)$ be a  pair, where   $M$ is a smooth manifold and $X$ is a smooth submanifold in $M$. 
The   embedding of $X$ in $M$ is denoted by   $i:X\to M$. 

Given an operator $D$  acting on functions on $M$, we define its trace on $X$ as
\begin{equation}\label{eq-43}
 i^* Di_*:H^s(X)\longrightarrow H^{s-d-\nu}(X),
\end{equation} 
where $ i^*:H^{s}(M)  \to H^{s-\nu/2}(X)$ is the boundary operator induced by the embedding   $i:X\to M$, while
$i_*:H^{s}(X)\to H^{s-\nu/2}(M)$ is the coboundary operator conjugate to  $i^*$, $d$ is the order of $D$, $\nu$ 
is the codimension of   $X$ in  $M$.  We denote the trace by   $i^!(D)$.

The aim of this paper is to study traces of $G$-operators   on   $M$. This class of operators is
associated with a smooth action of a compact Lie group   $G$ on $M$. Then a {\em $G$-operator } 
is  defined by an integral of the form   
\begin{equation}\label{eq-65}
D=\int_G D_gT_g dg,
\end{equation} 
where  $T_g$ stands for the shift operator
$$
T_gu={g^{-1}}^*u 
$$  
induced by  a diffeomorphism $g\in G$, $dg$ is the Haar measure on $G$, $\{D_g\}$ is a smooth family of pseudodifferential 
operators ($\psi$DO below) parametrized by $g\in G$.

The notion of trace in the smooth theory, i.e., in the situation, when the group is trivial, appeared
in the works \cite{NoSt1,NoSt2,Ster1}. There it was shown, in particular, that
the trace of a pseudodifferential operator is a pseudodifferential operator on   $X$. This fact made it possible to apply pseudodifferential operators
in the theory of Sobolev problems.  

If a nontrivial group $G$ acts on $M$ and we consider    $G$-operators \eqref{eq-65}
(see \cite{SaSt22,Ster20,Sav11}), the situation is completely different.  Namely, the trace
of a $G$-operator on a submanifold is an operator localized at a certain submanifold  in  $X$ (for instance,
it can be localized at a point, see \cite{SaSt36,Losh1}). Moreover,  the trace is not  a pseudodifferential operator
and is an operator of fundamentally new nature, for instance, it can be localized at a point(this can happen at the isolated fixed point of the group action), and to describe such operators, one needs not only Fourier transform (as for $\psi$DO's),
but also Mellin transform.  

The present paper is devoted to studying this new class of operators.  The main result is a localization theorem,
 which describes the set, on which the operator is localized. This theorem is
given in Section~\ref{par3}.  In subsequent sections we consider a number of examples,  which clarify
the situation both from the geometric  
(Section~\ref{par4}) and analytic (Section~\ref{par5}) points of view. Namely, we consider the situation  
when   $X$ is invariant with respect to the group action (in this case the trace of
a $G$-operator is again a $G$-operator on $X$), we also consider noninvariant submanifolds $X$, 
in which case the trace is localized on a lower-dimensional submanifold.

The work was partially supported by RFBR (projects~15-01-08392 and 16-01-00373), by Deutsche Forschungsgemeinschaft   (DFG), and also the Ministry of education and science of Russian Federation, agreement N~02.a03.21.0008. 
We are also grateful to the referee for useful remarks on a preliminary version of the paper.

\section{Localization theorem}\label{par3}

As we stressed in the introduction, the trace of a  $G$-operator $D$ on a submanifold $X\subset M$ is   generally   
localized at a certain subset in   $X$.  To describe this set, we consider the following two closed subsets in  $X$: 
\begin{enumerate}
\item The set of points  whose orbits are contained in   $X$
$$
X_G=\{x\in X \;|\; Gx\subset X\},\quad \text{where $Gx$ denotes the orbit of $x$};
$$
\item The set of points  whose orbits touch $X$
$$
\widetilde X_G= \{x\in X \;|\; T_x(Gx)\subset T_xX\}.
$$
Obviously, we have $  X_G\subset \widetilde X_G$.
\end{enumerate}

Below we assume that the following condition is satisfied:
\begin{condition}\label{cond1}
For each closed subset  $Z\subset  \widetilde X_G\setminus X_G$ there exists a family of open subsets
  $U_\varepsilon\subset X$, which contract to  $Z$ as $\varepsilon\to 0$, and such that the volume of the open set in $G$
$$
G_\varepsilon=\{g\in G \;|\; gU_\varepsilon\cap X\ne \emptyset\}
$$
 tends to zero as $\varepsilon\to 0$.
\end{condition}

Condition~\ref{cond1} is easy to check in examples (see below).  The meaning of this condition
is clear:   an element $g\in G$ takes the neighborhood  $U_\varepsilon$  to a set disjoint with   $X$ for almost all $g$
 (roughly speaking, if an orbit is not contained in   $X$,  then it is outside of  $X$ almost everywhere).

\begin{definition}
An  operator  $A$ is {\em localized on a subset} $Y\subset X$, if all compositions
$A\varphi$ are compact, whenever $\varphi$ is a smooth function   equal to zero in a neighborhood of $Y$.
\end{definition}

\begin{theorem}\label{prop1}{(on localization)}
The trace of a  $G$-operator $D$ on a submanifold $X$ is localized on the subset $X_G\subset X$.
In particular, if   $X_G$ is empty, then the trace is a compact operator.  
\end{theorem}

\begin{proof}

1.  Consider the operator
\begin{equation}\label{eq-tr1}
 Di_*=\int_G    D_gT_g i_* dg.
\end{equation}
Note that there is no restriction operator in this expression. We claim that this operator (and hence the trace  
 $i^*Di_*$) is localized on the subset  $\widetilde X_G$.

To prove this statement, we fix an arbitrary point  $x_0\in X\setminus\widetilde X_G $ and show that the operator 
 \eqref{eq-tr1} is compact on the subspace of functions supported in a neighborhood of   $x_0$. 
 To this end, we decompose the integral   \eqref{eq-tr1} over  $G$ into a finite sum of integrals over small neighborhoods
 and move the shift operator at an arbitrary point in each neighborhood outside   the integral sign. This enables us without loss of
  generality to pass to an operator of the form  
\eqref{eq-tr1}, where the integration is carried out  over a small neighborhood of the identity in $G$.
Further, we identify this neighborhood  with a neighborhood $U$ of zero in the Lie algebra denoted by   $\mathcal{G}$.
By the assumption,  $x_0\in X\setminus\widetilde X_G $, i.e., at $x_0$ the orbit  $Gx_0$ is not tangent to   $X$.  Hence, there
exists a vector   $h\in \mathcal{G}$   such that the corresponding vector field on   $M$ is not tangent to  $X$ at  $x_0$  (and also in a neighborhood of this point   by continuity , see Fig.~\ref{fig1}). 

\begin{figure}
 \begin{center}
  \includegraphics[width=6cm]{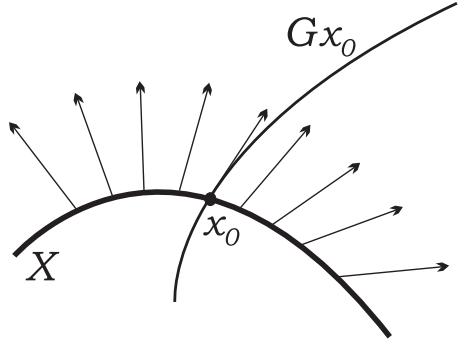} \end{center}\caption{A neighborhood of a point $x_0$,  at which the orbit is not tangent to $X$.}\label{fig1}
\end{figure} 

Then the integral \eqref{eq-tr1} (recall that this integral is actually only over $U$) can be considered as an iterated integral:
first along $h$ and then over directions transverse to  $h$.  Let us show that the integral  in the direction of $h$ 
\begin{equation}\label{eq-opa1}
\int _{-\varepsilon}^\varepsilon D'_a T_{\exp(ah)}i_*da,\qquad \text{ where }D'_a\equiv D_{\exp a}\text{ for simplicity},
\end{equation}
is a compact operator. 
\begin{lemma}
The operator \eqref{eq-opa1} acts continuously in the spaces   $H^s(X)\to H^{s-(d+(\nu-1)/2)}(M)$.
In particular, it is compact as an operator of order   $d+\nu$.
\end{lemma}
\begin{proof}The idea of the proof is that integration along   $h$ is an integration in a transverse direction to   $X$, 
and, when we perform it, there disappears    {\em one }  of the $\delta$-functions, which enter in the coboundary operator,
hence, we get the summand $-1/2$ in the  formula given above.  Let us give a detailed proof.

1. In a neighborhood of $x_0$ we introduce coordinates $x,y,t$ in $M$, in which  $X$ 
is determined by the equations $y=0,t=0$,  while the action of an element $\exp(ah)$ with $|a|$ small  on $M$  is equal to  $\exp(ah)(x,y,t)=(x,y,t+a)$.
The operator \eqref{eq-opa1} is denoted by $A$ for simplicity and we write it explicitly as  
\begin{multline}\label{eq-m1}
Au= \int_{-\varepsilon}^\varepsilon D'_a u(x)\delta(y)\delta(t-a)da=\\
=\iiint e^{i(x\xi+y\eta+t\tau)}\Bigl[\int_{-\varepsilon }^\varepsilon e^{-i\tau a}\sigma(D'_a)(x,y,t,\xi,\eta,\tau)da\Bigr] \widetilde u(\xi)d\xi d\eta d\tau.
\end{multline}
Here the coordinates $\xi,\eta,\tau$ are dual to $x,y,t$, while  $\sigma(D_a')$ is the symbol of the pseudodifferential operator $D_a'$, and
$ \widetilde u(\xi)$ stands for the Fourier transform of $u(x)$

2. Given $u\in H^s(X)$, we have to show that $Au\in H^{s'}(M)$, where $s'=s-d-\nu/2+1/2$. 
Indeed, integrating by parts the expression in the square brackets in   \eqref{eq-m1}, we obtain   the estimate 
\begin{equation}\label{est1}
\left|\int_{-\varepsilon }^\varepsilon e^{-i\tau a}\sigma(D_a)(x,y,t,\xi,\eta,\tau)da\right|\le C (1+|\tau|)^{-1}(1+|\xi|+|\eta|+|\tau|)^{d}.
\end{equation}
Here $C$ is some constant. Hence, we have
\begin{multline}
 \|Au\|^2_{s'}\le C \iiint\widetilde u(\xi)^2(1+|\tau|)^{-2}(1+|\xi|+|\eta|+|\tau|)^{2d+2s'} d\xi d\eta d\tau\le \\
\le C \iint\widetilde u(\xi)^2(1+|\tau|)^{-2}(1+|\xi|+|\tau|)^{2d+2s'+\nu-1} d\xi  d\tau\le C \int\widetilde u(\xi)^2(1+|\xi|)^{2s} d\xi=\|u\|^2_s.
\end{multline}
Here the first inequality follows from the properties of the Fourier transform, the second is obtained by integration over $\eta$ and using the spherical coordinates $\eta=(1+|\xi|+|\tau|)r\omega$; the third inequality is obtained using the change of variables $1+|\tau|=|\xi|p$ as follows
\begin{multline}
\int_{\mathbb{R}}(1+|\tau|)^{-2}(1+|\xi|+|\tau|)^{2d+2s'+\nu-1}  d\tau=
2|\xi|^{-2+2d+2s'+\nu-1+1}\int^\infty_{|\xi|^{-1}}  p^{-2} (1+p)^{2d+2s'+\nu-1}dp\\
=C|\xi|^{s-1}\times
\left(\begin{array}{ll} |\xi| & \text{if }|\xi|>1 \vspace{1mm} \\ |\xi|^{1-s} & \text{if }|\xi|<1 \end{array}\right)\le
 C(1+|\xi|)^{2s}.
\end{multline}
The proof of the lemma is now complete.
\end{proof}

2.  Let us show that the trace  $i^*Di_*$ is indeed supported on the subset   $X_G\subset \widetilde X_G$. 
To this end, we consider an arbitrary function  
$\varphi\in C^\infty(X)$, which vanishes in a neighborhood of $X_G$. We claim that the operator 
$$
i^*Di_*\varphi:H^s(X)\longrightarrow H^{s-d-\nu}(X)
$$ 
is compact. Indeed, let us use Condition~\ref{cond1} and take $Z$ equal to $\widetilde X_G\cap \supp \varphi$. 

Consider the decomposition
\begin{equation}\label{eq-tr2}
i^*Di_*\varphi =\int_{G_\varepsilon}  i^*D_g T_g i_*\varphi  dg+\int_{G\setminus G_\varepsilon} i^*D_g T_g i_*\varphi  dg.
\end{equation}
Here the first integral has a small norm (since $\vol( G_\varepsilon)\to 0$ as $\varepsilon\to 0$), 
while the second integral is, on the one hand,  (by   Item 1   above) supported on $Z $, 
and  on the other hand, on this set it is compact by locality  
(indeed, if  $x\in U_\varepsilon$ and $g\in G\setminus G_\varepsilon$, then $gx\in M\setminus X$, 
hence $ i^*D_g T_g i_*\varphi $ is compact).   As $\varepsilon\to 0$,  it  follows from the decomposition
\eqref{eq-tr2} into a sum of an operator with small norm and a compact operator that the operator on the left hand side of the equality
is compact.  

The proof of the localization theorem is now complete.  
\end{proof}

\section{Examples (geometry)}\label{par4}

Let us give examples of manifolds and group actions and apply the localization theorem (Theorem~\ref{prop1}) to them.

\paragraph{Rotations of  lines in the plane.} $M=\mathbb{R}^2, X=\mathbb{R}^1,$ while $G=\mathbb{S}^1$ is the group
of rotations around the origin. Here there are two possibilities, depending on whether the line 
  $X$ passes through the origin or not. If the line passes through the origin, then the traces of operators
  are supported at the fixed point   $X_G=\widetilde X_G=\{A\}$ 
(see Fig.~\ref{fig2}, Item 1). If the line does not pass through the fixed point,
then the traces of operators are all compact.   
More precisely, each such line has a unique point  $\widetilde X_G=A$, whose orbit touches $X$, 
while there are no fixed points: $X_G=\emptyset$ (see Fig.~\ref{fig2}, Item~2). Meanwhile, Condition~\ref{cond1} is 
satisfied (as $U_\varepsilon$ we can take an interval of length $\varepsilon$ with center at $A$).

\paragraph{Shifts of the plane.} $M=\mathbb{R}^2, X=\mathbb{S}^1,$ while $ G=\mathbb{R}$ is the group of translations
(see Fig.~\ref{fig2}, Item~3). This example is similar to the previous one, namely, the orbits are tangent to  $X$ at the points denoted by $A$ and $B$,
while there are no fixed points: $\widetilde X_G=\{A\}\cup \{B\}$, $X_G=\emptyset$. Hence, traces of operators 
on   $X$ are compact  (Condition~\ref{cond1} is satisfied in this case). On Fig.~\ref{fig2}, Item~4 we consider
and example, in which our Condition~\ref{cond1} is not satisfied
(here $X_G=\emptyset$, while  $\widetilde X_G$ consists of the horizontal intervals). 
\begin{figure}
 \begin{center}
  \includegraphics[width=10cm]{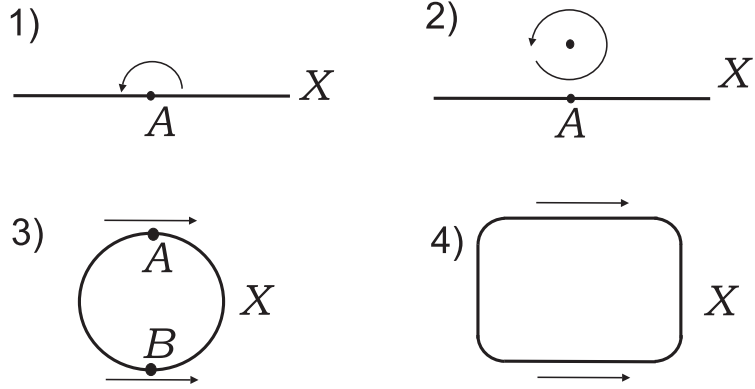} \end{center}\caption{Rotations and shifts of the plane.}\label{fig2}
\end{figure} 

\paragraph{Rotations of lines in $\mathbb{R}^3$.} $M=\mathbb{R}^3, X=\mathbb{R}^1,$ while $ G=\mathbb{S}^1$ is the group of rotations
about the  $OZ$-axis, while the line  $X$ and the axis of rotation intersect and form angle $\alpha$ 
(see Fig.~\ref{fig3}, Item~1).  In this case traces of $G$-operators are localized at the 
intersection of the two lines, except the case, when   $\alpha=0$ (more precisely, for $\alpha\ne 0$ we have $\widetilde X_G=X_G=\{A\}$).

\paragraph{Rotations of circles in $\mathbb{R}^3$.} $M=\mathbb{R}^3, X=\mathbb{S}^1,$ while $ G=\mathbb{S}^1$ is the group
of rotations about the $OZ$-axis  and the circle $X$ touches the axis of rotation (see Fig.~\ref{fig3}, Item 2).   
In this case the trace of operators is localized at the point, where the circle meets the axis of rotation.  This
degenerate case is interesting, since the trace is also localized at a one-point set and there arizes the question about the nature of this operator.  

\paragraph{Rotations of planes in $\mathbb{R}^3$.} $M=\mathbb{R}^3, X=\mathbb{R}^2,$   $ G=\mathbb{S}^1$ is the group
of rotations around the $OZ$-axis, while  the normal to the plane   $X$ and the axis of rotation 
form an angle  equal to  $\alpha$ (see Fig.~\ref{fig3}, Item~3). 
In this case   traces of $G$-operators are localized at the point of intersection
of the plane with the axis of rotation   
(more precisely, $\widetilde X_G=$ is the line $l$, while  $X_G=\{A\}$). 
\begin{figure}
 \begin{center}
  \includegraphics[width=14cm]{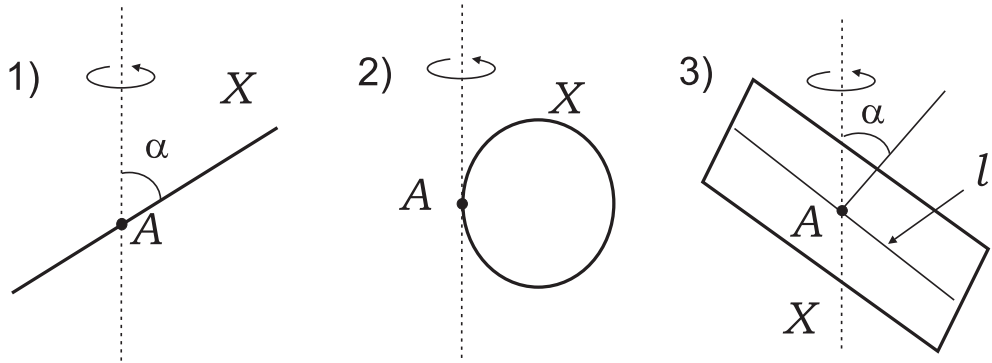} \end{center}\caption{Rotations in $\mathbb{R}^3$.}\label{fig3}
\end{figure} 

\paragraph{Rotations and shifts of spheres in $\mathbb{R}^3$.} $M=\mathbb{R}^3, X=\mathbb{S}^2,$ while $ G=\mathbb{S}^1\times\mathbb{R}^1$ is the group generated by rotations about OZ and shifts along this axis. The sphere
 $X$ has center at the origin  (see Fig.~\ref{fig4}, Item~1).  
 In this case   traces of $G$-operators are   compact.  Indeed, the orbits touch
the spheres at the equator $\widetilde X_G=\mathbb{S}^1$, while none of the orbits
is contained in the sphere and thus  $X_G=\emptyset$ 
(Condition~\ref{cond1} is satisfied in this case).

\paragraph{Arbitrary rotations of a line in $\mathbb{R}^3$.}  $M=\mathbb{R}^3, X=\mathbb{R}^1,$ while $G=SO(3)$ is the group of all rotations about the origin. Here  the line   $X$ passes through the origin.  
In this case   traces of $G$-operators are localized at the origin  (see Fig.~\ref{fig4}, Item~2).  
\begin{figure}
 \begin{center}
  \includegraphics[width=8cm]{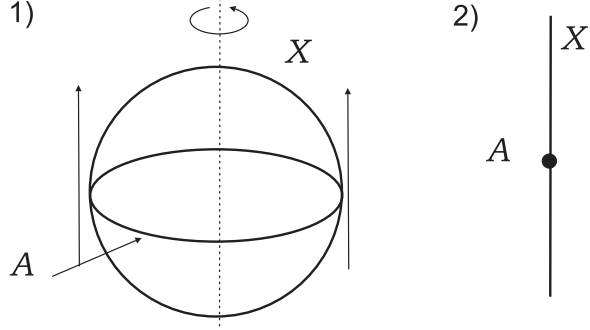} \end{center}\caption{Actions of Lie groups 1) $G=\mathbb{S}^1\times\mathbb{R}$ and 2) 
  $SO(3)$ in $\mathbb{R}^3$.}\label{fig4}
\end{figure}

\section{Examples (analysis)}\label{par5}

In this section we consider examples and
study the nature of traces of $G$-operators in a neighborhood of the set, on which they are localized.   

\paragraph{Example 1. } Let a submanifold  $X\subset M$ be $G$-invariant, i.e. we have $X=X_G=\tilde X_G$. 
In this case traces of  $G$-operators on $M$ are localized
on the entire manifold  $X$ by Theorem~\ref{prop1}.  It turns out, that in this invariant situation 
the trace is actually, a
$G$-operator on $X$ with respect to the restriction of the action of  $G$ on
$X$ and the corresponding shift operators
$$
T'_g : H^s(X)\longrightarrow H^s(X), \qquad T'_gu(x)=u(g^{-1}x).
$$
More precisely, the following statement is true. 
\begin{proposition}
 The trace of a $G$-operator $D$ (see~\eqref{eq-65}) on a $G$-invariant submanifold   $X$ is a $G$-operator on $X$,
i.e., we have
\begin{equation}\label{eq-d1}
 i^!(D)=\int_G D'_g T'_g dg,
\end{equation}
where $D'_g$ is a smooth family of pseudodifferential operators on $X$.
\end{proposition}
\begin{proof}
Let us compute the trace of the integrand in \eqref{eq-65}. We have 
$$
i^!(D_gT_g)u(x)=i^*D_g T_g (u(x)\otimes \delta_X) =i^* D_g (T'_g u(x)\otimes\delta_X)=
(i^*D_g i_*)(T'_g u(x))= i^!(D_g)T'_g u(x).
$$
Here we used the fact that $G$ is compact and chose a   $G$-invariant $\delta$-function supported on $X$.
Hence, integration over   $G$ gives the desired equality \eqref{eq-d1}, where $D'_g=i^!(D_g)$ is a pseudodifferential 
operator on the submanifold.
 
The proof of the proposition is now complete. 
\end{proof}

Of course,  a submanifold is not $G$-invariant in general. Below we study traces in such situations.  

\paragraph{Example 2. }

Let $X\subset \mathbb{R}^3$ be the plane
$$
-x\sin\alpha+z\cos\alpha=0.
$$
A basis in this plane is given by  the vectors     $e_1=(\cos\alpha,0,\sin\alpha)$, $e_2=(0,1,0)$. A normal vector
is equal to  $e_3=(-\sin\alpha,0,\cos\alpha).$
The coordinates in the basis $e_1,e_2,e_3\in \mathbb{R}^3$ are denoted by $(u,v,w)$. These coordinates are related
with the coordinates   $x,y,z$ as
\begin{equation}\label{change1}
(x,y,z)=( u\cos\alpha- w\sin\alpha, v, u\sin\alpha +w\cos\alpha ).
\end{equation} 

In $\mathbb{R}^3$, we consider the scalar $G$-operator
\begin{equation}\label{eq-op3}
D=\Delta^{-1}\int_{\mathbb{S}^1}T_\varphi d\varphi,
\end{equation} 
where $\Delta$ stands for the Laplacian, while $T_\varphi $ is the shift
operator associated with the group of rotations about the  
$OZ$-axis:
$$
(T_\varphi f)(x,y,z)=f (x\cos\varphi-y\sin\varphi,x\sin\varphi+y\cos\varphi,z).
$$
Let us study the trace
\begin{equation}\label{eq-op3a}
i^!(D)=  i^* \left(\Delta^{-1}\int_{\mathbb{S}^1}T_\varphi d\varphi \right) i_*:H^s(X)\to H^{s+1}(X),\quad s\in(-1,0),
\end{equation} 
of $D$ on $X$.

\begin{proposition}
The trace of operator \eqref{eq-op3} is localized at the point $(0,0,0)$ of intersection of the plane $X$ with
the axis of rotation. 
\end{proposition}
\begin{proof}
In this case a direct computation shows that   
$  \widetilde{X}_{\mathbb{S}^1}=\{y=0, -x\sin\alpha+z\cos\alpha=0 \}$ is a line,
while  ${X}_{\mathbb{S}^1}=\{(0,0,0)\}$ is a point. Hence, by Theorem~\ref{prop1}
the trace is localized at $(0,0,0)$. 
\end{proof}
So, the trace is localized at the fixed point of the group action. Let us study its structure by freezing
the coefficients of the operator at this point. Below, it will be more convenient to 
work with zero-order operators. We make reduction to this case by taking products of our operator with appropriate
powers of the Laplacian   $\Delta_X$ on $X$, i.e., from operator \eqref{eq-op3}  we pass to the operator
\begin{equation}\label{eq-77}
\Delta_X^{1/2} i^*\left( \Delta^{-1}\int_{\mathbb{S}^1}T_\varphi d\varphi\right) i_*
 :H^s(\mathbb{R}^2)\to H^s(\mathbb{R}^2).
\end{equation}

A direct computation shows that in the dual space with respect to the Fourier transform operator  \eqref{eq-77} 
is written as an integral operator as
\begin{multline}\label{eq1}
f(s,t)\longmapsto (u^2+v^2)^{1/2}\int_{\mathbb{R}}dw\int_0^{2\pi} \frac{d\varphi}{ u^2+v^2+w^2}\times\\
f\Bigl(u(\cos^2\alpha\cos\varphi+\sin^2\alpha)+v\cos\alpha\sin\varphi+w\sin\alpha\cos\alpha(1-\cos\varphi),
-u\cos\alpha\sin\varphi+v\cos\varphi+w\sin\alpha\sin\varphi\Bigr):\\
: \widetilde H^s(\mathbb{R}^2_{s,t})\longrightarrow \widetilde H^s(\mathbb{R}^2_{u,v}).
\end{multline}
Here we used the fact that operators on the physical space are transformed to the following operators on the dual space:  
\begin{itemize}
\item the coboundary operator  $i_*$ is transformed to the operator 
$$
\pi^*f(x,y,z)=f(x\cos\alpha +z\sin\alpha ,y)
$$
where $\pi:\mathbb{R}^3\to \mathbb{R}^2$ denotes the projection;
\item  the rotation operator ${T}_\varphi$  is transformed to the rotation operator
$$
\widetilde{T}_\varphi f(x,y,z)=f (x\cos\varphi+y\sin\varphi,-x\sin\varphi+y\cos\varphi,z);
$$
\item the boundary operator $i^*$ is transformed to the operator of integration with respect to   $w$:
$$
\pi_*f(u,v)=\int_\mathbb{R} f( u\cos\alpha- w\sin\alpha, v, u\sin\alpha +w\cos\alpha )dw
$$
(here we used the change of variables  \eqref{change1});
\item replacing the operators in the composition \eqref{eq-77} by the corresponding operators on the dual space, we
obtain precisely the operator   \eqref{eq1};
\item the space $\widetilde H^s(\mathbb{R}^2_{u,v})$ is the closure of the set of smooth compactly-supported functions with respect to the norm  
$$
\|f\|_s^2=\int_{\mathbb{R}^2} |f(u,v)|^2 (1+u^2+v^2)^s dudv.
$$
\end{itemize}

Below we use polar coordinates on   $X$: 
$$
s=\rho \cos\psi, t=\rho \sin\psi \quad \text{and} \quad u=r\cos\omega, v=r\sin\omega.
$$

In the integral in~\eqref{eq1} we make the following change of variables   $(w,\varphi)\mapsto (\rho,\psi)$:
\begin{equation}\label{eq-change1}
\begin{array}{llll}
u(\cos^2\alpha\cos\varphi+\sin^2\alpha)& +v\cos\alpha\sin\varphi&+w\sin\alpha\cos\alpha(1-\cos\varphi) &= 
\rho\cos\psi,\\
-u\cos\alpha\sin\varphi&+v\cos\varphi&+w\sin\alpha\sin\varphi&=\rho\sin\psi.
\end{array}
\end{equation}
It turns out that this change of variables is one-to-one, while the inverse change is equal to  
(cumbersome computations are omitted):
\begin{equation}\label{eq-inv1}
  \begin{array}{ll}
 \displaystyle  \tg\frac\varphi 2 &\displaystyle= \frac{\rho\cos\psi -u}{(v+\rho\sin\psi)\cos\alpha},\vspace{1mm}\\
\displaystyle   w&\displaystyle=u\ctg\alpha-\frac{v\ctg\varphi}{\sin\alpha}+\frac{\rho\sin\psi}{\sin\alpha\sin\varphi}=\vspace{1mm}\\
&\displaystyle =\frac{u^2(1-2\cos^2\alpha)-2\rho u \cos\psi\sin^2\alpha+\rho^2(\cos^2\psi+\sin^2\psi\cos^2\alpha)-v^2\cos^2\alpha}{2(\rho\cos\psi-u)\cos\alpha\sin\alpha}=\vspace{1mm}\\
&\displaystyle =\frac{\sin^2\alpha(\rho\cos\psi-u)^2+\cos^2\alpha(\rho^2-r^2)}{2(\rho\cos\psi-u)\cos\alpha\sin\alpha}.
\end{array}
\end{equation}
Let us describe the geometric meaning of this change of variables. In the   $(s,t)$-plane
we have an ellipse (defined parametrically in terms of $\varphi$):
$$
s=u(\cos^2\alpha\cos\varphi+\sin^2\alpha)+v\cos\alpha\sin\varphi, \quad t=-u\cos\alpha\sin\varphi+v\cos\varphi.
$$
As  $\varphi$ increases, the corresponding point goes around the ellipse clockwise, starting from the point
$(u,v)$ (see Fig.~\ref{fig5}). Further, for each point of this ellipse  (i.e., for a given   $\varphi$) equations
\eqref{eq-change1} define a line in the  $(s,t)$-plane, which passes through this point with 
parameter   $w$ along the line (this line degenerates to a point for $\varphi=0$). A direct computation shows
that this line passes through the above mentioned point of the ellipse and 
the   point with the coordinates   $(u,-v)$.
\begin{figure}
 \begin{center}
  \includegraphics[width=9cm]{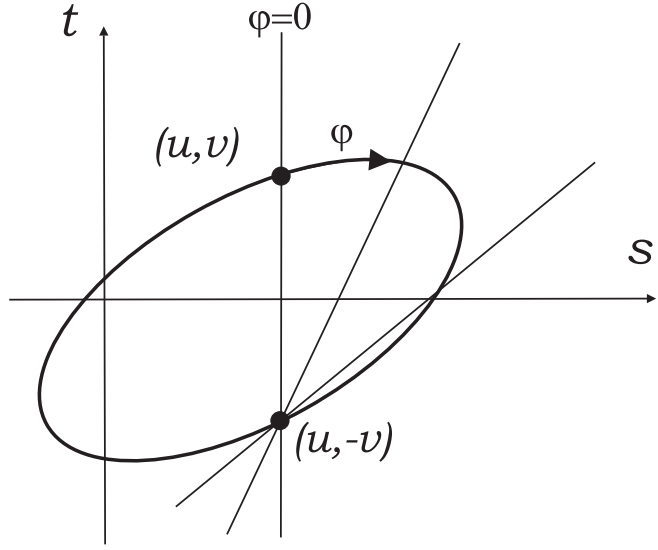} \end{center}\caption{Change of variables $(w,\varphi)\mapsto (\rho,\omega)$.}\label{fig5}
\end{figure} 

Clearly, the mapping  $(w,\varphi)\mapsto (\rho,\omega)$ is a diffeomorphism, except at the points,
which lie on the vertical line passing through the   point   $(u,-v)$ of the ellipse.  

Making the change of variables  \eqref{eq-change1}, we rewrite
the product of the differentials in  \eqref{eq1} in the new coordinates as
\begin{equation}\label{eq-dd1}
d\varphi d w=\frac{\rho d\rho d\psi}{\cos\alpha\sin\alpha ((1-\cos\varphi)(w\sin\alpha -u\cos\alpha)+v\sin\varphi)}=
\frac{\rho d\rho d\psi}{ \sin\alpha ( \rho\cos\psi-u)}.
\end{equation}
Substituting \eqref{eq-change1} and \eqref{eq-dd1} in the integral  operator \eqref{eq1}, 
we rewrite this integral operator as  
\begin{equation}\label{eq2}
f(s,t)\longmapsto  
  r\int_0^\infty d\rho\int_0^{2\pi} \frac{\rho d\psi}{ (r^2+w^2) \sin\alpha |\rho\cos\psi-u|}
f(\rho\cos\psi,\rho\sin\psi)
=\int_0^\infty  K( \rho/r)f(\rho)\frac {d\rho}\rho,
\end{equation}
where $K(\rho)$ is a family of integral operators on the circle equal to
\begin{multline}\label{eq3}
  (K(\rho)f)(\omega)=\int_0^{2\pi} \frac {\rho^2 }{(1+w^2) \sin\alpha |\rho\cos\psi-\cos\omega|}  f(\psi)d\psi=\\
=\int_0^{2\pi} \frac {4\sin \alpha \cos^2\alpha  \rho^2 |\rho\cos\psi -\cos\omega|}{ 
4  \cos^2\alpha\sin^2\alpha(\rho\cos\psi -\cos\omega)^2+\bigl(\sin^2\alpha(\rho\cos\psi-\cos\omega)^2+\cos^2\alpha(\rho^2-1)\bigr)^2 }  f(\psi)d\psi.
\end{multline}
Since the operator-function $K(\rho/r)$ in \eqref{eq2} 
is homogeneous of degree zero with respect to the pair of its arguments, operator  \eqref{eq2} 
is nothing but the Mellin convolution in the variable $r$. Hence, it is algebraized 
if we apply the Mellin transform   $\mathcal{M}_{\rho\to p }$. Here by algebraization we mean that the operator is
written as an operator of multiplication by the function
\begin{equation}\label{eq-mm1}
\widehat K(p)=\mathcal{M}_{\rho\to p }K(\rho)=\int_0^\infty \rho^p K(\rho)\frac{d\rho}\rho.
\end{equation} 
Necessary properties of this operator-function are described in the following two lemmas.  
\begin{lemma}\label{lem1}
The operator-function $K(\rho)$ ranges in integral operators with smooth kernel for all $\rho>0$ and $\rho\ne 1$, 
and its operator norm in the space   $L^2(\mathbb{S}^1)$   has the following estimates
\begin{equation}\label{eq-33}
\|K(\rho)\|=\left\{
\begin{array}{ll}
 O(\rho^2), & \text{ if }\rho<1/2,\vspace{1mm}\\
O\left( |\rho-1|^{-1/2} \right),& \text{ if }1/2<\rho<2,\vspace{1mm}\\
 O(\rho^{-1}), & \text{ if }\rho>2.
\end{array}
\right.
\end{equation}
\end{lemma}
\begin{proof}
1. The singularities of the Schwarz kernel of the operator  \eqref{eq3} correspond to the zeroes
of the denominator. Since this denominator is a sum of squares,
the singularities of the denominator are determined from the equations  
$$
\rho\cos\psi-\cos\omega=0,\qquad \rho^2-1=0,
$$
which are equivalent to  $\rho=1,\psi=\pm\omega$. This implies that for  $\rho\ne 1$ the denominator has no zeroes,
hence, the Schwarz kernel is smooth. The first statement in the lemma is now proved. 

2. Estimates of the integral kernel and the operator norm as $\rho\to\infty$ and $\rho\to 0$ are obtained similarly.
Namely, as $\rho\to\infty$ the numerator in \eqref{eq3} is equal to $O(\rho^3)$ and the denominator
has  a lower bound $\ge C\rho^4$, which gives us the desired estimate. Finally, as $\rho\to 0$ 
the numerator is of order $O(\rho^2)$, while the denominator is separated from zero, which gives the desired estimate.  

3. It remains to estimate the norm of operator   $K(\rho)$ as $\rho\to 1$. So, we consider  $\rho$ close but not equal to
$1$. To estimate the norm of the integral operator    $K(\rho)$, we use the Schur test   (e.g., see \cite{HaSa1}) and estimate the integrals
$$
\int_{\mathbb{S}^1} |K(\rho,\omega,\psi)|d\omega,\qquad  \int_{\mathbb{S}^1} |K(\rho,\omega,\psi)|d\psi
$$
of the kernel  $K(\rho,\omega,\psi)$ uniformly in $\omega,\psi$. Let us estimate the first of the integrals
 (the second is estimated similarly). We have
\begin{multline}\label{eq-st1}
\int_{\mathbb{S}^1} |K(\rho,\omega,\psi)|d\omega\le  C\int_{\psi-\varepsilon}^{\psi+\varepsilon}
\frac{|\rho\cos\psi -\cos\omega|d\omega}{(\rho\cos\psi -\cos\omega)^2+ \bigl(\tg^2\alpha(\rho\cos\psi-\cos\omega)^2+(\rho^2-1)\bigr)^2}\le \\
\le  C\int_{-\varepsilon_1}^{\varepsilon_2}\frac{|(\rho-1)\cos\psi -t|dt}
{
\bigl[((\rho-1)\cos\psi -t)^2+ \bigl(\tg^2\alpha((\rho-1)\cos\psi -t)^2+(\rho^2-1)\bigr)^2\bigr] \sqrt{|\sin^2\psi-2t\cos\psi-t^2|} } 
\end{multline}
(for some numbers $\varepsilon_1,\varepsilon_2\ge 0$, which are bounded uniformly in $\psi$ and $\rho$). In the 
last inequality in \eqref{eq-st1} we reduced our integral to an integral
over a small neighborhood of  the point $\psi$, since the integrand is uniformly bounded
whenever   $|\omega\pm\psi|>\varepsilon$ and is an even function. Then in the second inequality we made the change 
of variable   $\omega\mapsto t$:
$$
\cos\omega=\cos\psi+t,\qquad d\omega=\frac{\pm dt}{\sqrt{|\sin^2\psi-2t\cos\psi-t^2|}}.
$$
Then in the latter integral in   \eqref{eq-st1} we make the change of variable $t=|\rho-1|\tau$; 
then we obtain
\begin{multline}\label{eq-st2}
\int_{\mathbb{S}^1} |K(\rho,\omega,\psi)|d\omega\le  \\
\le  C\int_{-\varepsilon_1|\rho-1|^{-1}}^{\varepsilon_2|\rho-1|^{-1}}\frac{|\cos\psi \mp \tau| }
{\bigl[(\cos\psi \mp \tau )^2+ \bigl(\tg^2\alpha(\rho-1)(\cos\psi \mp \tau)^2+(\rho+1)\bigr)^2 \bigr]}\times\\
\times \frac{d\tau}{
\sqrt{|\sin^2\psi-2\tau|\rho-1|\cos\psi-\tau^2|\rho-1|^2|} } \le \\
\le  C \int_{-\varepsilon_1|\rho-1|^{-1}}^{\varepsilon_2|\rho-1|^{-1}}
\frac{d\tau}{(|\tau|+1)\sqrt{|(|\rho-1|\tau+\cos\psi+1)(|\rho-1|\tau+\cos\psi-1)|}}
\end{multline}
The second inequality here follows from the fact that the numerator is     $O(|\tau|+1)$, while
the expression in square brackets in the denominator is nonzero, and $\ge \tau^2$ at infinity. 
The last integral in \eqref{eq-st2} admits the estimate   
$$
\le  C \int_{-\varepsilon_1|\rho-1|^{-1}}^{\varepsilon_2|\rho-1|^{-1}}
\frac{d\tau}{ |\tau|+1 }\le C \ln|\rho-1|^{-1},  
$$
provided that $|\cos\psi\pm 1|>\varepsilon_1,\varepsilon_2$.
Let us now consider the case, when one of the numbers  $  \cos\psi\pm 1 $ is small. For definiteness, we
consider the case, when   $\psi$ is close to zero (the case, when $\psi$ is close to $\pi$ is considered similarly). 
Then we have the following estimate of the integral in  \eqref{eq-st2}
\begin{multline}\label{eq-st3}
\le C \int_{-\varepsilon_1|\rho-1|^{-1}}^{\varepsilon_2|\rho-1|^{-1}}
\frac{d\tau}{(|\tau|+1)\sqrt{| |\rho-1|\tau+\cos\psi-1 |}}=\\
=
\frac C{\sqrt{|\rho-1|}}\int_{-\varepsilon_1|\rho-1|^{-1}}^{\varepsilon_2|\rho-1|^{-1}}
\frac{d\tau}{(|\tau|+1)\sqrt{\displaystyle\left| \tau+\frac{\cos\psi-1}{|\rho-1|}\right| }}\le
 \frac C{\sqrt{|\rho-1|}}.
\end{multline}

Thus, we obtain the estimate
$$
\int_{\mathbb{S}^1} |K(\rho,\omega,\psi)|d\omega= O(|\rho-1|^{-1/2}) \text{ as $\rho \to 1$ }.
$$
Similarly, we obtain 
$$
\int_{\mathbb{S}^1} |K(\rho,\omega,\psi)|d\psi= O(|\rho-1|^{-1/2}).
$$
These estimates and the Schur test \cite{HaSa1} give the desired norm estimate of the integral operator  
$$
\|K(\rho)\|=O(|\rho-1|^{-1/2}).
$$

The proof of the lemma is now complete.
\end{proof}

\begin{lemma}\label{lem2}
The operator-function $\widehat K(p)$ (see~\eqref{eq-33}) enjoys  the following properties
\begin{enumerate}
\item[1)]  it is holomorphic for all $p$ in the vertical strip 
$$
  \{ -2<\re p<1\} \subset \mathbb{C};
$$ 
\item[2)]  it ranges in integral oeprators on   $\mathbb{S}^1$ with smooth kernel;
\item[3)]  as $\im p\to\infty$  in the above described vertical strip, we have $\|\widehat K(p)\|\to 0$.
\end{enumerate}
\end{lemma}
\begin{proof}
Indeed, for all $p$ in the vertical strip the integral \eqref{eq-mm1} converges,
since the integral of norms  
$$
 \int_0^\infty |\rho^{p-1}|\cdot  \|K(\rho)\| d\rho 
$$
is absolutely convergent by the estimates \eqref{eq-33}. The remaining statements of the lemma follow
from well-known properties of the Mellin transform.  
\end{proof}

We are now ready to describe the structure of the trace  \eqref{eq-op3a}.  
Namely, we showed that this trace is localized at the fixed point and after applying  Fourier  and then Mellin transform in the radial variable in the dual space, the operator reduces to an operator
of multiplication by the function  $\widehat K(p)$. Gathering these transformations, we obtain a representation of the trace  \eqref{eq-op3a} in the form
\begin{equation}\label{eq-63}
 \Delta_X^{-1/2} \chi \mathcal{F}^{-1}  \chi' \mathcal{M}^{-1}_{p\to\rho} \widehat K(p) \mathcal{M}_{\rho\to p}\chi' \mathcal{F}\chi:H^s(X)\longrightarrow H^{s+1}(X).
\end{equation} 
Here $\mathcal{F}$ is the Fourier transform, $\mathcal{M}_{r\to p}$ 
is the Mellin transform in the radial variable in the dual space, while the cut-off functions   $\chi$, $\chi'$ are written to obtain a bounded operator in the corresponding function spaces. In more detail,  $\chi$ is a smooth function on $X$  
  equal to zero outside a small neighborhood of zero and is identically equal to one  in a small neighborhood of zero, while  $\chi'$ is a function in the dual space identically equal to one at infinity and zero in a neighborhood of zero. The trace \eqref{eq-op3a} and the operator \eqref{eq-63} are equal up to compact summands by the locality principle.  


\paragraph{Example 3}

On the product $\mathbb{R}^2_{x,z}\times\mathbb{S}^1_y$, we consider the  action 
$$
g_\varphi(x,y,z)= (x\cos\varphi+z\sin\varphi,y+\varphi,-x\sin\varphi+z\cos\varphi),\qquad \varphi\in\mathbb{S}^1
$$
of the group $\mathbb{S}^1_\varphi$ (this action consists of screw motions: shifts along $y$ by $\varphi$ 
and rotations in the   $XOZ$-plane by angle $\varphi$).

Let us study the trace of the  $G$-operator
$$
D=\Delta^{-1}\int_{\mathbb{S}^1}T_\varphi d\varphi,\qquad\text{ where } T_\varphi u(x,y,z)=u(g^{-1}_\varphi(x,y,z)),
$$
on the submanifold  equal to the  horizontal coordinate plane:
$ 
X=\{z=0\}.
$ 

By the localization theorem the  trace 
\begin{equation}\label{tr-1}
i^!(D)=i^* \left(\Delta^{-1}\int_{\mathbb{S}^1}T_\varphi d\varphi\right)i_*:H^{s}(X)\to H^{s+1}(X) ,\quad s\in (-1,0),
\end{equation}
is localized at the submanifold   $X_{\mathbb{S}^1}=\{x=z=0\}\subset X$ equal to the the   $OY$-axis, about which we make rotations.

Then we represent $H^s(X)$   as the space of sections of a Hilbert bundle over $X_{\mathbb{S}^1}$ with fiber
  $H^s(\mathbb{R})$ and we denote this bundle by  $\mathcal{H}^s(X_{\mathbb{S}^1})$.  

We represent  the shift operator as  the composition
$$
T_\varphi=T'_\varphi T''_\varphi
$$ 
of the shift $T'_\varphi $ along $X_{\mathbb{S}^1}$ and a rotation $T''_\varphi$ in the $XOZ$-plane.

The structure of the trace \eqref{tr-1} is described in the following proposition.
\begin{proposition}
The trace  \eqref{tr-1} is a $G$-operator with operator-valued
symbol on   $X_{\mathbb{S}^1}$ modulo compact summands. More precisely, the trace can be written as  
\begin{equation}\label{tr-2}
i^!(D)= \int_{\mathbb{S}^1}\left( i^*\Delta^{-1}T''_\varphi i_*\right) T'_\varphi d \varphi:\mathcal{H}^{s}(X_{\mathbb{S}^1})\to \mathcal{H}^{s+1}(X_{\mathbb{S}^1}),
\end{equation}
where the operator in brackets is a family of pseudodifferential operators  on $X_{\mathbb{S}^1}$ 
with operator-valued symbols. Moreover, the operators in the family continuously depend
on  $\varphi$ in operator norm for all $\varphi\ne 0,\pi$ and the norms of the operators
and their symbols are uniformly bounded for all    $\varphi$.
\end{proposition}
\begin{remark}
Note that pseudodifferential operators with operator-valued symbols of nonzero order were introduced in \cite{NSScS99}.
\end{remark}

\begin{remark}
The operator family $i^*\Delta^{-1}T''_\varphi i_*$ is not norm continuous at   $\varphi=0$ and $\pi$. 
This is easy to see, since for small    $\varphi\ne 0$ the corresponding operator
is localized at   $X_{\mathbb{S}^1}\subset X$, while for $\varphi=0$ it is a pseudodifferential operator
and, hence, is localized on the entire submanifold  $X$. Similarly, one shows
that there is no limit in norm as   $\varphi\to \pi$.
\end{remark}

\begin{proof}
A direct computation shows that  \eqref{tr-2} holds.
Let us show that the operator in round brackets in  \eqref{tr-2} is a pseudodifferential operator
with operator-valued symbol on   $X_{\mathbb{S}^1}$. {Indeed, writing this
operator as  
$$T''_\varphi i^*_\varphi\Delta^{-1} i_*,
$$ 
where $i_\varphi:g_\varphi X\to \mathbb{R}^2\times\mathbb{S}^1$
stands for the shift of the initial submanifold  $X$ by an element $g_\varphi$, one 
can show that this operator is a $\psi$DO, since the
factor  $T''_\varphi$ acts as identity along the base $X_{\mathbb{S}^1}$, while  $ i^*_\varphi\Delta^{-1} i_*$ 
is a translator and, as shown in~\cite{SaSt32}, is a $\psi$DO}. 

The boundedness of norms of these operators follows from their definition.
To prove the norm boundedness of the operator family  $ i^*\Delta^{-1}T''_\varphi i_*$,
and also obtain uniform boundedness of the symbols, we calculate the symbol   (as an operator in the space dual with 
respect to the Fourier transform 
of functions depending on $\xi,\eta$). This operator is equal to 
\begin{multline}
 u(\xi,\eta)\mapsto  \int_{\mathbb{R}} \frac{u(\xi\cos\varphi-\zeta\sin\varphi,\eta)d\zeta}{\xi^2+\eta^2+\zeta^2}=  \int_{\mathbb{R}} \frac{u(z,\eta)dz}{|\sin\varphi|\left(\xi^2+\eta^2+\left(\frac{\xi\cos\varphi-z}{\sin\varphi}\right)^2\right)}=\\
=  |\sin\varphi|\int_{\mathbb{R}} \frac{u(z,\eta)dz}{\xi^2-2\xi z \cos\varphi+z^2+\eta^2\sin^2\varphi}.
\end{multline}
Here we made the change of variable $z=\xi\cos\varphi-\zeta\sin\varphi$ in the integral.
We denote the latter symbol by   $A_\varphi(\eta)$.  It smoothly depends on 
  $\eta$ and is  {\em twisted homogeneous} (e.g., see \cite{NSScS99}) with respect to this variable:
$$
A_\varphi(\lambda \eta)=\lambda^{-1} \varkappa^{-1}_\lambda A_\varphi(\eta) \varkappa_\lambda,
\qquad \text{ for all }\lambda>0,
$$
where $\varkappa_\lambda f(z)= f(\lambda z)$ denotes the action of the group $\mathbb{R}_+$ of dilations. 

It remains to show that the symbol remains bounded as  $\varphi\to 0$.
By twisted homogeneity and unitarity of the group   $\varkappa_\lambda$, we assume that $\eta=1$ and obtain 
$$
A_\varphi(1) u= |\sin\varphi|\int_{\mathbb{R}} \frac{u(z)dz}{(\xi-z\cos\varphi)^2+\sin^2\varphi\langle z\rangle ^2}.
$$
Here we use the notation $\langle x\rangle =(1+x^2)^{1/2}$.

The commutative diagram
\begin{equation}
 \xymatrix{
     L^2(\mathbb{R}_z,\langle z\rangle ^{2s}) \ar[r]^{A_\varphi(1)} \ar[d]_{\langle z\rangle ^{s}}&  L^2(\mathbb{R}_z,\langle z\rangle ^{2(s+1)})  \ar[d]_{\langle z\rangle ^{s+1}}\\
L^2(\mathbb{R}_z) \ar[r]   &L^2(\mathbb{R}_z)
}
\end{equation}
imples that the norm of  $A_\varphi(1)$ is equal to the $L^2$-norm of the operator
  $\langle \xi\rangle ^s A_\varphi \langle z^{-s}\rangle $. So, it remains to estimate the norm of the integral operator
   in $L^2(\mathbb{R})$:
$$
u(z)\longmapsto \int_{\mathbb{R}} \frac{\langle \xi\rangle ^{s+1}|\sin\varphi|\langle z\rangle ^{-s}u(z)dz}{(\xi-z\cos\varphi)^2+\sin^2\varphi\langle z\rangle ^2}.
$$
The kernel of this integral operator is denoted by  $K(\xi,z)$. To estimate
the norm of this integral operator, we use the Schur test. To this end, let us obtain the uniform boundedness
of the integrals   
\begin{equation}\label{eq-mu4}
\int_\mathbb{R} |K(\xi,z)|dz, \qquad \int_\mathbb{R} |K(\xi,z)|d\xi.
\end{equation}
Let us estimate the first integral (the second is estimated similarly). 

We have
\begin{multline}
\int_\mathbb{R} |K(\xi,z)|dz=\int_{\mathbb{R}} \frac{\langle \xi\rangle ^{s+1}|\sin\varphi|\langle z\rangle ^{-s} dz}{(z-\xi\cos\varphi)^2+\sin^2\varphi\langle \xi\rangle ^2} 
=\int_{\mathbb{R}} \frac{dt}{t^2+1}\langle \xi\rangle ^{s } \Bigl(\langle \xi\cos\varphi+|\sin\varphi|\langle \xi\rangle t \rangle \Bigr)^{-s} =\\
=\int_{\mathbb{R}} \frac{dt}{t^2+1}
\left(\frac 1{\langle \xi\rangle }+\left(\frac \xi{\langle \xi\rangle }\cos\varphi+|\sin\varphi|t\right)^2\right)^{-s/2}\le\\
\le    \int_{\mathbb{R}} \frac{dt}{t^2+1} (1+(1+|t|)^2)^{-s/2} < \infty.
\end{multline}
Here in the second equality we made the change of variable   $z=\xi\cos\varphi+|\sin\varphi|\langle \xi\rangle t$  in the integral; while   the first inequality follows from the fact that  $x^{-s}$ is increasing for $s<0$;  the latter integral converges, since  $2+s< 1$.

Estimates of the norm of the second integral in~\eqref{eq-mu4} are obtained along the same lines.

So, by the Schur test the norm of the symbol is uniformly bounded. The continuity with respect to the operator
norm follows from the fact that, as is easy to see,   the symbol is differentiable 
with respect to the parameter   $\varphi$ if $\sin\varphi\ne 0$. Hence, the operator with this symbol
continuously depends on this parameter in the operator norm.  
\end{proof}

\newpage


\end{document}